\documentclass[12pt]{amsart}
\usepackage[margin=1in]{geometry}
\usepackage{amsthm}
\usepackage{amsmath,xypic}
\usepackage{epigraph}
\usepackage{graphicx}
\usepackage{natbib}
\usepackage{enumitem}
\newcommand{\be}{\begin{equation}}
\newcommand{\ee}{\end{equation}}
\newcommand{\bes}{\begin{equation*}}
\newcommand{\ees}{\end{equation*}}
\newcommand{\bea}{\begin{eqnarray}}
\newcommand{\eea}{\end{eqnarray}}
\newcommand{\beas}{\begin{eqnarray}}
\newcommand{\eeas}{\end{eqnarray}}
\newcommand{\ben}{\begin{note}}
\newcommand{\een}{\end{note}}
\newcommand{\bexl}{\vskip0.1em\noindent\hrulefill\vskip1em\begin{ExerciseList}}
\newcommand{\eexl}{\end{ExerciseList}\hrulefill}

\newcommand{\bth}{\begin{theorem}}
\newcommand{\eth}{\end{theorem}}
\newcommand{\bpro}{\begin{prop}}
\newcommand{\epro}{\end{prop}}
\newcommand{\bp}{\begin{proof}}
\newcommand{\ep}{\end{proof}}
\newcommand{\blem}{\begin{lemma}}
\newcommand{\elem}{\end{lemma}}
\newcommand{\bn}{\begin{note}}
\newcommand{\en}{\end{note}}
\newcommand{\br}{\begin{remark}}
\newcommand{\er}{\end{remark}} 
\newcommand{\bed}{\begin{defn}}
\newcommand{\eed}{\end{defn}}

\newcommand{\bcon}{\begin{conj}}
\newcommand{\econ}{\end{conj}}

\newtheorem{theorem}[equation]{Theorem}      
\newtheorem{lemma}[equation]{Lemma}          %
\newtheorem{corollary}[equation]{Corollary}  
\newtheorem{proposition}[equation]{Proposition}

\theoremstyle{definition}
\newtheorem{conj}[equation]{Conjecture}

\theoremstyle{definition}
\newtheorem{defn}[equation]{Definition}
\theoremstyle{remark}

\theoremstyle{definition}
\newtheorem{remark}[equation]{Remark}

\numberwithin{equation}{section}

\let\into=\hookrightarrow

\let\tensor=\otimes

\newcommand{\C}{{\mathbb C}}

\newcommand{\End}{\rm{End}}

\newcommand{\Pic}{{\rm Pic\,}}
\newcommand{\Q}{{\mathbb Q}}

\newcommand{\spec}{{\rm Spec}}

\newcommand{\Z}{{\mathbb Z}}

\renewcommand{\int}{\operatorname{int}}
\renewcommand{\O}{{\mathcal O}}

\newcommand{\alb}{{\rm Alb}}
\newcommand{\et}{\acute{e}t}

\newcommand{\fm}{{\mathfrak{M}}}

\let\fm=\fa

\let\mr=\mapright

\renewcommand{\bpro}{\begin{proposition}}
\renewcommand{\epro}{\end{proposition}}

\begin{document}

\title[]{On varieties with trivial tangent bundle}%
\author{Kirti Joshi}%
\address{Math. department, University of Arizona, 617 N Santa Rita, Tucson
85721-0089, USA.} \email{kirti@math.arizona.edu}

\thanks{}%


\begin{abstract}
I prove a crystalline characterization of abelian
varieties in characteristic $p>0$ amongst the class of varieties
with trivial tangent bundle.  I show using my characterization that a smooth, projective, ordinary variety with trivial tangent bundle is an abelian variety if and only if its second crystalline cohomology is torsion free. I also show that a conjecture of Ke-Zheng Li about characteristic $p>0$ varieties with trivial tangent bundles is true for surfaces. I give a new proof of a result of Li and prove a refinement of it. Based on my characterization of abelian varieties I propose modifications  of Li's conjecture which I expect to be true.
\end{abstract}
\maketitle
\epigraph{And here I stand, with all my lore,\\
Poor fool, no wiser than before.}{Goethe, Faust part I}


\section{Introduction}
Let $k$ be a field and let $X$ be a smooth projective variety over
$k$. For $k=\C$ it is well-known, and elementary to prove, that if
$X$ has trivial tangent bundle, then $X$ is an abelian variety. In \citep*{igusa55} it was shown that this is false in characteristic $p>0$. \citep*{mehta87}  studied ordinary varieties with trivial tangent bundle and proved that they
have many properties similar to abelian varieties, including
the Serre-Tate theory of canonical liftings. In Theorem~\ref{thm:characterisation}, I present two equivalent \textit{crystalline characterizations} of abelian varieties amongst the class of varieties with trivial tangent bundle. My characterization is the following: a smooth, projective variety $X$ with trivial tangent bundle is an abelian variety if and only if it has a smooth Picard scheme and it satisfies Hodge symmetry in dimension one (I call such a variety \textit{Picard-Hodge Symmetric}, see Def~\ref{def:picard}). Another equivalent characterization is given in terms of what I call \textit{minimally Mazur-Ogus varieties} (see Def~\ref{def:minimal}). A smooth, projective variety is a minimally Mazur-Ogus variety if $H^2_{cris}(X/W)$ is torsion-free and Hodge de Rham spectral sequence degenerates in dimension one.  In Corollary~\ref{cor:lifting} I show that any smooth, projective variety with trivial tangent bundle which lifts to $W_2$ and with $H^2_{cris}(X/W)$ torsionfree is an abelian variety. In Remark~\ref{rmk:relaxed2} I discuss a natural question raised by Li in his emails to me about weakening the hypothesis of Theorem~\ref{thm:characterisation}.

In \citep*[Conjecture 4.1]{li10} it is conjectured that if $p>3$ then every smooth,
projective variety with trivial tangent bundle is an abelian
variety. I show in Theorem~\ref{thm:surface-case} that this conjecture is true in dimension two.

In dimension two, the most famous example of a surface in characteristic $p=2$ with trivial tangent bundle and which is not an abelian variety is due to \citep*{igusa55} (Igusa surface for $p=2$ has been studied by many authors including Torsten Ekedahl; for a recent treatment of the Igusa surface see \citep*{chai-igusa}). Let me note that the Igusa surface of characteristic $p=2$ also has a less well-known cousin in characteristic $p=3$.

I observe in Theorem~\ref{thm:igusa-var2}  that if $p=2$ then for every $g\geq 2$ and for every $1\leq r<g$, there is a family of varieties dimension $g$ with trivial tangent bundle and which are not  abelian varieties. This family is parameterized by $\fm_r^{\rm ord}[p]\times\fm_{g-r}$ where $\fm_g$ is the  moduli stack of abelian varieties of dimension $g$ and the superscript `{\rm ord}' stands for the ``ordinary locus'' and $\fm_r^{\rm ord}[p]$ is the moduli stack of ordinary abelian varieties with a point of order $p$. For $p=3$ one has a slightly weaker result--see Theorem~\ref{thm:igusa-var3}.

In Remark~\ref{rem:optimal}, I note that the two conditions:  minimally Mazur-Ogus, Picard-Hodge symmetry in Theorem~\ref{thm:characterisation} cannot be weakened or relaxed.  In general, presence of torsion in crystalline cohomology and non-degeneration of Hodge de Rham are not correlated conditions.

In Theorem~\ref{thm:fsplit}, I show  that a smooth, projective, ordinary variety with trivial tangent bundle is an abelian variety if and only if its second crystalline cohomology is torsion free.

In \citep*[Theorem~4.2]{li10} (also see \citep*{li91})  it is shown that if $p>2$ and $X$ is ordinary with trivial tangent bundle then $X$ is an abelian variety.
In Theorem~\ref{th:li-newproof}, I give a new proof of Li's remarkable theorem \citep*[Theorem~4.2]{li10} and in fact I prove a sharpening of \citep*[Theorem~4.2]{li10} and of \citep{mehta87}. I show that for $p=2$ any smooth, projective, ordinary  variety with trivial tangent bundle has a minimal Galois \'etale cover (see Def. \ref{def:minimal-cover}) by an abelian variety with Galois group of exponent $p=2$. Li's approach is based on infinitesimal group actions while I use Serre-Tate canonical liftings (of abelian varieties) and the theory of complex multiplication and its influence on the slopes of Frobenius (see \citep{yu04}).

In the light of my characterization (Theorem~\ref{thm:characterisation}), especially as torsion in the second crystalline cohomology  can occur for any prime $p$,  it seems to me that perhaps the original conjecture of Li (see \citep*[Conjecture 4.1]{li10}) needs to be modified. In fact there are two different versions of Li's conjecture which I conjecture. The first version is the fixed characteristic version which says that there exists an integer $n_1(p)$ such that if $X$ is any variety  of dimension less than $n_1(p)$, with trivial tangent bundle over an algebraically closed field of characteristic $p>0$ is an abelian variety (see Conjecture~\ref{li-mod-p}).

The fixed dimension version (see Conjecture~\ref{li-mod-dim}), inspired by \citep*{liedtke09}, says that for any fixed integer $d\geq 2$.  There exists an integer $n_0(d)$ such that  any smooth, projective variety $X/k$ with dimension $d$ and with trivial tangent bundle is an abelian variety if $p>n_0(d)$. (Clearly for $d=1$, one has $n_0(1)=1$; for $d=2$, $n_0(2)=3$ by Theorem~\ref{thm:surface-case}).

I would like to thank Vikram Mehta for bringing \citep*{li10} to my attention and for many conversations around Li's conjecture. Many years ago (around 1991-92) Vikram had explained to me his paper with V.~Srinivas (see \citep*{mehta87}), and more recently during my visit to India in 2012 we had many lucid conversations on topics of common interest (we were studying some questions on fundamental group schemes). I was amazed by his incredible energy and zest for mathematics while  facing an illness which ultimately took him from us.

I proved Theorem~\ref{thm:characterisation} while I was on a visit to RIMS, Kyoto in 2011. I thank RIMS, Kyoto for providing excellent hospitality and I am grateful to Shinichi Mochizuki providing me with an opportunity to visit RIMS. I thank KeZheng Li for answering many of my elementary and naive questions about his papers \citep*{li91,li10}, and for his comments and corrections. I thank Brian Conrad for pointing out \citep{yu04}.

\section{Characterization of abelian varieties}
In this section I give a crystalline characterization of abelian varieties in the class of smooth, projective varieties with trivial tangent bundle. My characterization requires  the following two definitions.

\bed\label{def:minimal}
Let $X$ be a smooth, projective variety over  an algebraically closed field $k$ with $\mathrm{char}(k)=p>0$. I say that $X$ is a \emph{minimally Mazur-Ogus variety} if $X$
satisfies the following two conditions:
\begin{enumerate}
  \item $H^2_{cris}(X/W)$ is torsion free,
  \item Hodge to de Rham sequence degenerates in dimension one.
\end{enumerate}
\eed

\br
Conditions underlying Mazur-Ogus varieties were introduced  in \citep*{ogus79} where a number of their properties are studied, the nomenclature, I believe, is due to Torsten Ekedahl.  A smooth, projective variety $X$ is a Mazur-Ogus variety if $H^*_{cris}(X/W)$ is torsion-free and Hodge de Rham spectral sequence degenerates.
\er

\bed\label{def:picard}
Let $X$ be a smooth, projective variety over  an algebraically closed field $k$ with $\mathrm{char}(k)=p>0$. I say that $X$ is a \emph{Picard-Hodge symmetric variety} if it
satisfies the following two conditions:
\begin{enumerate}
  \item Picard scheme of $X$ is smooth,
  \item Hodge symmetry holds in dimension one.
\end{enumerate}
\eed

The main theorem of this section is the following  characterization theorem alluded to in the Introduction.

\begin{theorem}\label{thm:characterisation}
Let $X$ be any smooth, projective variety with trivial tangent
bundle. Then the following are equivalent.
\begin{enumerate}
  \item $X$ is minimally Mazur-Ogus,
  \item $X$ is Picard-Hodge symmetric,
  \item $X$ is an abelian variety.
\end{enumerate}
\end{theorem}

\begin{proof}
Let us prove $(1)\Rightarrow(2)\Rightarrow(3)\Rightarrow(1)$. Let us
begin with $(1)\Rightarrow(2)$. Assume that $X$ is minimally
Mazur-Ogus. The fact that $H^2_{cris}(X/W)$ is torsion-free implies
$\Pic(X)$ is reduced  (see \citep*{illusie79b}) and by the
universal coefficient theorem for crystalline cohomology one sees that
\be
	H^1_{cris}(X/W)\tensor_W k\mr{\sim}H^1_{dR}(X/k).
\ee
As Hodge to de Rham degenerates in dimension one, one sees that
\be
	\dim(H^1_{dR}(X/k))=h^{0,1}+h^{1,0}.
\ee
Reducedness of Picard variety means that
\be
	\dim(H^1_{cris}(X/W)\tensor_W k)=2h^{0,1}
\ee and the degeneration of
Hodge de Rham mean that
\be
	2h^{0,1}=h^{1,0}+h^{0,1}.
\ee
Thus one sees that
\be
	h^{1,0}=h^{0,1}.
\ee
Putting all this together one sees that $X$ is
Picard-Hodge symmetric variety. Thus one sees that
(1)$\Rightarrow$(2).

Now I prove (2)$\Rightarrow$(3). Suppose that $X$ is Picard-Hodge symmetric
variety and $X$ has trivial tangent bundle so $H^0(X, \Omega^1_X)$
has dimension $n=\dim(X)$. As $X$ is Picard-Hodge symmetric one sees
that
\be
	h^{0,1}=h^{1,0}=\dim(X).
\ee
Thus $\dim(\Pic(X))=\dim(X)$ and by hypothesis of (2)
$\Pic(X)$ is reduced. Hence the Picard variety is also the Albanese variety: $\Pic^0(X)=\alb(X)$ and in particular $$\dim(X)=\dim(\Pic(X))=\dim(\alb(X)).$$ Let $X\to\alb(X)$ be the Albanese morphism. By \citep*[Lemma 1.4]{mehta87} one sees that the
Albanese morphism $X\to\alb(X)$ is a smooth surjective morphism with connected fibres and
$\Omega^1_{X/\alb(X)}=0$. So $X\to\alb(X)$ is a finite, surjective \'etale morphism with connected fibres and hence it is an isomorphism.

Now it remains to prove that (3)$\Rightarrow$(1). This is standard (see \citep*{illusie79b}).
\end{proof}

The following corollary of \citep{deligne87} and Theorem~\ref{thm:characterisation} is immediate as one has Hodge de Rham degeneration in dimensions $\leq p-1$ for any $p$ (and hence in dimension one for any $p\geq 2$).

\begin{corollary}\label{cor:lifting}
Let $X/k$ be a smooth, projective variety with trivial tangent bundle. Suppose $X$ satisfies the following:
\begin{enumerate}
\item $H^2_{cris}(X/W)$ is torsionfree,
\item $X$ lifts to $W_2$.
\end{enumerate}
Then $X$ is an abelian variety.
\end{corollary}

\br
Let me point out that for  the Igusa surface ($p=2,3$), $H^2_{cris}(X/W)$ is not torsion-free (but Hodge-de Rham degenerates in dimension one) and Hodge symmetry is true in dimension one, but $\Pic(X)$ is not reduced.
\er

\br\label{rmk:relaxed2} 
In his recent email to me, KeZheng Li has suggested that perhaps, any smooth, projective variety with trivial tangent bundle and reduced Picard scheme is an abelian variety. This is certainly natural expectation. I include some comments on this question.

Firstly let me point out that there are two important numbers $\dim(X)=\dim H^0(X,\Omega^1_X)$ and $\dim(\Pic(X))=\dim H^1(X,\O_X)$ which must be equal if this assertion holds. On the other hand even if $\Pic(X)$ is reduced, it seems difficult to prove that these two numbers are equal without some additional crystalline torsion-freeness hypothesis. Note that the pull-back of one-forms on $\Pic(X)=\alb(X)$, by $X\to\alb(X)$, lands inside the subspace of closed one-forms $H^0(X,Z_1\Omega^1_X)$ and all of the following inclusions 
$$H^0(\alb(X),\Omega^1_{\alb(X)})\subset H^0(X,Z_1\Omega^1_X)\subset H^0(X,\Omega^1_X)$$ are strict in general. By \citep[Prop.~5.16, page 632]{illusie79b} the hypothesis that $H^2_{cris}(X/W)$ is torsionfree is equivalent to reducedness of $\Pic(X)$ and the equality $H^0(\alb(X),\Omega^1_{\alb(X)})=H^0(X,Z_1\Omega^1_X)$. In particular the second inclusion does not become equality even if we assume $H^2_{cris}(X/W)$ is torsionfree, and so it is not possible to work with a simpler hypothesis: $\Pic(X)$ is reduced at the moment. 

Secondly let me point out that the reducedness of Picard scheme controls only a part of the crystalline torsion which may arise in this situation. Torsion arising from non-reducedness of $\Pic(X)$ is of a fairly mild sort (``divisorial torsion'' in the terminology of \citep{illusie79b}). But Ekedahl has shown that the self-product of the Igusa type surface with itself carries exotic torsion in $H^3$. It is possible that a similar example (of dimension bigger than two)  exists in which $H^2_{cris}(X/W)$ has exotic torsion, since there is a plethora of examples (see Theorem~\ref{thm:igusa-var2}) in any dimension for $p=2$ and one can probably use deformation theory to provide examples with subtler torsion behaviour. 

So relaxing the conditions in Theorem~\ref{thm:characterisation} seems a bit too optimistic (to me) and at any rate requires a fuller understanding of the crystalline cohomology of varieties with trivial tangent bundles (which I do not possess).

It is possible to provide alternate formulations of Theorem~\ref{thm:characterisation}, but I have chosen formulations which are easiest to deal with in practice.
\er

\section{Surfaces with trivial tangent bundle}
Let $X/k$ be a smooth projective variety over an algebraically
closed field of characteristic $p>0$. The main theorem of this section is the following. This was conjectured by KeZheng Li in \citep*[Conjecture 4.1]{li10}.

\begin{theorem}\label{thm:surface-case}
Let $X/k$ be a smooth projective surface over an algebraically
closed field of characteristic $p>3$ and assume that the tangent
bundle $T_X$ of $X$ is trivial. Then $X$ is an abelian surface.
\end{theorem}
\begin{proof}
As $T_X=\O_X\oplus\O_X$ one sees that $\Omega^1_X=O_X\oplus\O_X$ and
so $\Omega^2_X=\O_X$. Thus $c_1(X)=0$ and also as $T_X$ is trivial
one sees that $c_2(X)=0$. Now it is immediate by the adjunction formula (see \citep*{hartshorne-algebraic}) that $X$ is a minimal surface  of Kodaira dimension $\kappa(X)=0$.

By Noether's formula $12\chi(\O_X)=c_1^2+c_2$ (see \citep*{hartshorne-algebraic}), one sees that
\be
	\chi(\O_X)=0.
\ee This means
\be
\chi(\O_X)=0=h^0-h^{0,1}+h^{0,2};
\ee
and as $K_X=\O_X$
by Serre duality one sees that $H^2(\O_X)=H^0(\O_X)$, and hence  that
\be
	h^{0,1}=2.
\ee
Next $c_2=0$ gives
\be
	c_2=b_0-b_1+b_2-b_3+b_4=2-2b_1+b_2=0.
\ee
Thus one sees that $b_1\neq 0$ and one has $b_2\neq 0$ because $X$ is
projective (the Chern class of any ample class is non-zero in
$H^2_{\et}(X,\Q_\ell)$). Now $b_1$ is even as $b_1$ is the Tate
module of the Albanese variety of $X$ (which is reduced by
definition). Thus one has $b_1\geq 2$.

Then by \citep*[Page 25]{bombieri77} one sees that there are exactly two
possibilities for the pair $(b_1,b_2)$: either $(b_1,b_2)=(4,6)$
or $(b_1,b_2)=(2,2)$. If one is in the first case, by classification
of \citep*[Page 25]{bombieri77} $X$ is an abelian surface.

If not, one is in the second case. In this case one has $b_1=2$ so
$q=1$ and $h^1(\O_X)=2$. Thus one sees that $\Pic(X)$ is non-reduced
and at any rate the surface $X$ is hyperelliptic and as $p>3$,
classification (see \citep*[Page 37]{bombieri77}) shows that the order
of $K_X$ must be one of $2,3,4,6$ which is at any rate $>1$. On the
other hand one has $K_X=\O_X$. Thus $X$ cannot be hyperelliptic.

So one sees that the second case cannot occur and $X$ is an abelian
surface as asserted.
\end{proof}

By  a \textit{family of varieties with trivial tangent bundle} I mean a proper, flat $1$-morphism of stacks $f:X\to M$, with $M$ a Deligne-Mumford stack (over schemes over $k$) such that $f$ is schematic and for every morphism of stacks $\spec(k')\to M$ with $k'\supset k$ a field, the fibre product $X\times_M\spec(k')$ is a geometrically connected, smooth, projective scheme over $k'$ with trivial tangent bundle.

The construction of Igusa surface (\citep*{igusa55}) leads to the following (for another variant of this construction see Proposition~\ref{prop:minimal-examples}). For $g\geq1$, let $\fm_{g}$ be  moduli stack of abelian varieties of dimension $g$ over $k$ (see \citep*{faltings-book,mumford-GIT}). Let $\fm_g^{\rm ord}$ be the dense open substack of ordinary abelian varieties in the moduli  stack of abelian varieties of dimension $g$ over $k$, more generally let $\mathcal{U}_g^{\rm \geq 1}[p]\subset\fm_g$ be the stack of abelian varieties  of dimension $g$, with a point of order $p$.

\bth\label{thm:igusa-var2}
Let $k$ be an algebraically closed field of characteristic $p=2$. Then for every $g\geq 2$, and for any $1\leq r<g$, there exists a family, parameterized by $\mathcal{U}_r^{\rm \geq 1}[p]\times \fm_{g-r}$ of smooth, projective varieties of dimension $g$ over $k$ which are not abelian varieties and with trivial tangent bundles. In particular there is a family  parameterized by $\fm_r^{\rm ord}[p]\times \fm_{g-r}$, of smooth, projective varieties of dimension $g$ over $k$ which are not abelian varieties and with trivial tangent bundles.
\eth

\bp
First let me recall the following version of Igusa's construction (see \citep{igusa55}). For additional variants of Igusa's construction see Proposition~\ref{prop:minimal-examples} below.
Let $B_1$ be an abelian variety of dimension $r$ over $k$ with $2$-rank at least one, let $t\in B_1[2](k)$ be a non-trivial two torsion point. Let $B_2$ be any abelian variety over $k$ of dimension $g-r$. Then consider the Igusa action on $A=B_1\times B_2\to B_1\times B_2$ given by $(x,y)\mapsto (x+t,-y)$. Then this gives an action of $\Z/2$ on $A$ which is fixed point free and
\be
	H^0(A,\Omega^1_A)^{\Z/2}= H^0(A,\Omega^1_A),
\ee
as $\Z/2$ acts by translation on the first factor and so acts trivially on one forms of $B_1$ and on the second factor the action on the space of one forms of $B_2$ is by $-1=1$ and hence is trivial on the space of one forms on the second factor as well. Let $X$ be the quotient of $A$ by this $\Z/2$ action. Then $T_X$ is trivial (as $H^0(X,T_X)=H^0(A,T_A)$). On the other hand by Igusa, ${\rm Alb}(X)=B_1/\langle t\rangle$ and so $X$ is not an abelian variety and $\Pic(X)$ is not reduced.

Now one simply has to note that one can carry out Igusa's construction on the universal abelian scheme over the moduli stack of abelian schemes (of the above sort).
\ep

For $p=3$ the result is a little weaker, by simply taking products with an abelian variety one gets:
\bth\label{thm:igusa-var3}
Let $p=3$ and $k$ be an algebraically closed field of characteristic $p$. Then for every  $g\geq 2$, there exists a family of smooth, projective varieties of dimension $g$ over $k$ which are not abelian varieties and with trivial tangent bundle.
\eth

\br\label{rem:optimal}
Note that for the Igusa surface one has $\dim H^0(X,\Omega^1_X)=\dim H^1(X,\O_X)$ so Hodge symmetry holds and Hodge-de Rham does degenerate at $E_1$ but $\Pic(X)$ is not reduced and $H^2_{cris}(X/W)$ has torsion. Varieties $X$, constructed as in  Theorem~\ref{thm:igusa-var2} from ordinary abelian varieties, have the property that they are ordinary with trivial tangent bundle;  one has lifting to $W_2$ (by \citep[Theorem~9.1]{joshi07} of V.~B.~Mehta) and hence Hodge-de Rham degenerates in dimension $<p$ (by \citep{deligne87}), but  $H^2_{cris}(X/W)$ is not torsion-free.  Thus these varieties are neither Picard-Hodge symmetric nor are they minimally Mazur-Ogus.
\er

\section{Ordinary varieties with trivial tangent bundle}\label{fsplit}
I give a proof of the following theorem.
\bth\label{thm:fsplit}
Let $X$ be a smooth, projective variety with trivial tangent bundle. Then the following are equivalent:
\begin{enumerate}
\item[$(1)$] $X$ is ordinary and minimally Mazur-Ogus,
\item[$(1')$] $X$ is ordinary and Picard-Hodge symmetric,
\item[$(2)$] $X$ is Frobenius split and minimally Mazur-Ogus,
\item[$(2')$] $X$ is Frobenius split and Picard-Hodge symmetric,
\item[$(3)$] $X$ is ordinary and $H^2_{cris}(X/W)$ is torsion-free,
\item[$(3')$] $X$ is Frobenius split and $H^2_{cris}(X/W)$ is torsion-free,
\item[$(4)$] $X$ is an ordinary abelian variety.
\end{enumerate}
\eth
\bp
The equivalences $(1)\iff (1')$ and $(2)\iff (2')$ are clear from the proof of Theorem~\ref{thm:characterisation}. The equivalence $(3)\iff (3')$ is \citep*[Lemma~1.1]{mehta87}. The equivalence (1) $\iff$ (2) is immediate from \citep*[Lemma~1.1]{mehta87} as $X$ is ordinary if and only if $X$ is Frobenius split. Now  (2) $\implies$ (3) is clear from the Definition~\ref{def:minimal} and by \citep*{mehta87}. Now  to prove (3) $\implies$ (4). This is immediate from Theorem~\ref{thm:characterisation}, provided one proves that Hodge de Rham spectral sequence degenerates in dimension $\leq 1$. In other words one has to show that the hypothesis of $(3)$ implies that $X$ is minimally Mazur-Ogus. This is proved as follows. Any smooth, projective variety with trivial tangent bundle is ordinary if and only if it is Frobenius split (see \citep*[Lemma~1.1]{mehta87}). A result of V.~B.~Mehta (see \citep*[Theorem~9.1]{joshi07}) says that a Frobenius split variety $X$ lifts to $W_2$ and hence Hodge de Rham degenerates in dimension $\leq p-1$ by \citep*[Corollaire~2.4]{deligne87}. Hence one has degeneration in dimension one for any $p\geq 2$. Hence hypothesis of (3) imply that $X$ is Mazur-Ogus. So the assertion (3) $\implies$ (4) follows from Theorem~\ref{thm:characterisation}. Now (4) $\implies$ (1) is standard (see \citep*{illusie79b}).
\ep

\begin{corollary}
Let $X$ be a smooth, projective, ordinary variety, with trivial tangent bundle. Then $X$ is an (ordinary) abelian variety if and only if  $H^2_{cris}(X/W)$ is torsion-free.
\end{corollary}

\section{New proof of Li's Theorem}
In this section I give a new proof of Li's Theorem (see \citep{li91,li10}) and prove the following refinement.
\begin{defn}\label{def:minimal-cover}
Let $X$ be a smooth, projective variety with trivial tangent bundle and suppose $A\to X$ is Galois an \'etale cover by an abelian variety. I say that $A\to X$ is a minimal Galois \'etale cover of $X$ is there is no factorization of $A\to X$ into \'etale morphisms $A\to A'\to X$ with $A'$ an abelian variety and $A'\to X$ Galois.
\end{defn}

Since an abelian variety $A$ has a non-trivial fundamental group, non-minimal \'etale covers exist if $G\neq\{1\}$.

\bth\label{th:li-newproof}
Let $k$ be an algebraically closed field of characteristic $p>0$. Let $X/k$ be a smooth, projective, ordinary variety with trivial tangent  bundle.
\begin{enumerate}
\item Either $X$ is an abelian variety, or
\item $p=2$ and $X$ has a minimal Galois \'etale cover by an abelian variety with Galois group of exponent $p$ (i.e every element is of order $p$).
\end{enumerate}
\eth
\bp
Let  $X$ be as in the statement of the theorem and suppose $X$ is not an abelian variety. By \citep{mehta87} there exists an ordinary abelian variety $A/k$ and a finite, Galois \'etale morphism $A\to X$ with Galois group $G$ of order a power of $p$ which acts freely on $X$. By passing to a quotient of $G$ if needed, one may assume that $A\to X$ is a minimal Galois \'etale cover of $X$. In particular $A$ carries fixed point-free automorphisms $\sigma:A\to A$ of order  $d=p^m$ a power of $p$. If $d=1$ for every element of $G$ then this is already case (1) so there is nothing to do; if $d=2$ for every element of $G$ then one is  in case (2) so again there is nothing to prove. So  assume $d=p^m\geq 3$ for some element  $\sigma\in G$.

Then, by \citep*[Lemma 3.3]{lange01} (the proof given there is characteristic free--and the argument is sketched below for convenience),  there are abelian varieties $A_1,A_2$ such that $A$ is isogenous to $A_1\times A_2$ and that $\sigma|_{A_1}$ is a translation, and $\sigma|_{A_2}$ is an automorphism (possibly with fixed points) of order a power of $d$. Indeed, write $\sigma=t_x\circ\sigma'$ where $t_x$ is a translation, $\sigma'$ an automorphism of order a power of $d$ and one may take $A_1$ to be the connected component of $\ker(1-\sigma')$ and $A_2={\rm image}(1-\sigma')$. As $A$ is ordinary, so are $A_1$ and $A_2$. One assumes, without loss of generality, that $\sigma'$ is a homomorphism of $A_2$. Now $(A_2,\sigma')$ admits a canonical Serre-Tate lifting to $W(k)$ (see \citep[Theorem~1(2) of Appendix]{mehta87}), and in particular a lifting $(B_2,\sigma')$ of $(A_2,\sigma')$ to complex numbers exists. So starting with $X$ one has arrived at an abelian variety $B_2$ over $W(k)$ and an automorphism  $\sigma': B_2\to B_2$ of finite order, with possibly finitely many fixed points. Replacing $B_2$ by a subabelian variety if needed, one may assume that $\sigma'$ is not a translation on any subvariety of $B_2$.

Now I proceed by an algebraic variant of \citep[Proposition 13.2.5 and Theorem 13.3.2]{lange-book}. This is done as follows. Let $\Phi_{d}(X)$ be the $d$--cyclotomic polynomial. So $\Phi_d(X)|(X^d-1)$ and $\Phi_d(X)$ is irreducible and the primitive $d$-th roots of unity are its only roots. Let $f$ be the endomorphism $\frac{{\sigma'}^d-1}{\Phi_d(\sigma')}$ of $B_2$, i.e., consider the polynomial $$f(X)=\frac{X^d-1}{\Phi_d(X)}\in\Z[X]$$ and consider the endomorphism $f:=f(\sigma'):B_2\to B_2$. Consider the subvariety $$B_3=f(B_2)\subset B_2.$$ Then $B_3$ is an abelian variety annihilated by $\Phi_d(\sigma')$ and hence is naturally a $\Z[\zeta_d]$-module. Moreover $B_3$ has good ordinary reduction at $p$, denoted $A_3$, and in particular $H^1_{dR}(B_3/W)=H^1_{cris}(A_3/W)$ is a $\Z[\zeta_d]\tensor_{\Z_p}W(k)$-module which is finitely generated and $\Z[\zeta_d]$-torsion free and hence projective of rank $k=2\dim(B_3)/\phi(d)$. Now every finitely generated projective module over $\Z[\zeta_d]$ of rank $k$ is a direct sum of ideals
$I_1\oplus I_2\oplus\cdots\oplus I_k$ of $\Z[\zeta_d]$. Using this  one sees that, up to isogeny, one may factor $B_3$ into product of $k$ abelian varieties $B_{3,1},\ldots,B_{3,k}$ each of dimension $\phi(d)/2$ (over $\C$ this is proved by an analytic argument, attributed to an unpublished result of S.~Roan, in \citep[Theorem~13.2.5]{lange-book}). Each of these varieties has (possibly up to isogeny) $\Z[\zeta_d]\into \End(B_{3,i})$ and as $2\dim(B_{3,i})=\phi(d)$, so each has complex multiplication by $\Z[\zeta_d]$. Fix one of these abelian varieties, say, $B_{3,1}$. Then by a basic result \citep[Theorem 3.1, page 8]{lang-CM-book} $B_{3,1}$ is isotypic with a simple abelian variety factor $B$ with complex multiplication by a CM subfield of $\Q(\zeta_d)$. Further $B$ has good ordinary reduction at $p$ (by virtue of its construction from $B_{3,1}$ which has ordinary reduction at $p$).

On the other hand note that $p$ is totally ramified in the cyclotomic field $\Q(\zeta_d)$ as $d=p^m\geq 3$, so $p$ is also totally ramified in the CM subfield for $B$. Hence one  sees, by \citep{yu04} or \citep[Prop. 3.7.1.6, Prop. 4.2.6]{conrad-book}, that the special fiber of $B$ at $p$  is  isoclinic of positive slope (equal to half). So it cannot be ordinary. This is a contradiction.

Thus  $d=p^m\leq 2$ and if $X$ is not an abelian variety  then one is in case (2). This completes the proof.
\ep

If $A$ is an abelian variety then $A$ acts on itself by translations. In particular translation by a non-trivial point of order $p$ is an automorphism of $A$ of order $p$. In what follows I say that an automorphism $\rho:A\to A$ is a \textit{non-trivial automorphism} if $\rho$ is not a pure translation.
Before proceeding let me point out the following variant of \citep{igusa55}.
\bpro\label{prop:minimal-examples}
For every algebraically closed field $k$ of characteristic $p=2$ or $p=3$, and for every $n\geq 1$ and for every integer $N>n$, there exists a smooth, projective variety $X/k$, of $\dim(X)=N$, with trivial tangent bundle and a minimal Galois \'etale cover with $G=(\Z/p)^n$.
\epro
\bp
Let $A,A_1, A_2,\ldots, A_{n}$ be abelian varieties over $k$  satisfying the following conditions:
\begin{enumerate}
\item Let $\rho_i:A_i\to A_i$, for $1\leq i\leq n$, be  a non-trivial automorphism of order $p$, such that for every $i$ the subspace of $\rho_i$-invariant one forms $H^0(A_i,\Omega^1_{A_i})^{\langle \rho_i\rangle}=H^0(A_i,\Omega^1_{A_i})$,
\item and one has  $\dim(A)+\dim(A_1)+\cdots+\dim(A_{n})=N$;
\item suppose $A$ has $p$-rank at least one.
\end{enumerate}

For $p=2$ any abelian varieties $A,A_1,\ldots,A_n$ satisfying the last two conditions satisfy the first with the automorphism $\rho_i:A_i\to A_i$ being $\rho_i(x)=-x$ for all $x\in A_i$ for $1\leq i\leq n$.  The condition on invariant forms is trivially satisfied as $-1=+1$ because $p=2$.

For $p=3$ consider an elliptic curve $E/k$ with a non-trivial automorphism of order $p=3$. Let $A_i=E$ for $1\leq i\leq n$.  The condition on invariants is trivially satisfied as $\Z/p=\Z/3$ operates unipotently on $H^0(E,\Omega^1_E)$ and as any unipotent action has a non-zero subspace of invariants and as $H^0(E,\Omega^1_E)$ is one dimensional, all one forms are invariant under this non-trivial automorphism of order three.

Thus for any $p=2,3$ one has abelian varieties satisfying all the three conditions.
Let $t\in A[p]$ with $t\neq 0$ be a point on $A$ of order $p$. Let $G=(\Z/p)^n$ and consider its elements  as vectors $(g,g_2,\ldots,g_{n-1})$ with entries in $\Z/p$ and let $G$ operate on $$B=A\times A_1\times A_2\times \cdots\times A_{n}$$ as follows: $$(1,g_2,\ldots,g_{n})\cdot(x,x_1,\ldots,x_{n})=(x+t,\rho_1(x_1),\rho_2^{g_2}(x_2),\ldots,\rho_n^{g_n}(x_n)),$$
and
with the usual convention $\rho_i^{0}=1$ (note the asymmetry in my notation and construction--this is intended to include Igusa surfaces for $n=1,N=2$). Then $G$ acts free of fixed points and the quotient $X=B/G$ is a smooth, projective variety with trivial tangent bundle with minimal \'etale cover with Galois group $G$ and $\dim(X)=N$.
\ep
\br
Let me give an example of an abelian variety $A$ in characteristic $p>3$ with $\dim(A)>1$ and a non-trivial automorphism $\rho:A\to A$ of order $p$, which shows that the condition on space of  invariants is not satisfied in general. Let $A$ be the Jacobian of the hyperelliptic curve $y^2=x^p-x$. Then the automorphism $(x,y)\mapsto (x+1,y)$ of $y^2=x^p-x$ is an automorphism of order $p$ (and hence of $A$). Using a standard basis for computing forms, one checks that the subspace of invariant forms is not of dimension equal to $\dim(A)$.
\er

\section{Variants of Li's Conjecture}
In \citep*[Conjecture 4.1]{li10} it was conjectured that for $p>3$ every smooth, projective variety with trivial tangent bundle is an abelian variety. Let me remark that the construction in Proposition~\ref{prop:minimal-examples}  also works for $p>3$, \textit{except} for the fact that I do not know how to construct abelian varieties satisfying the hypothesis on invariant forms in condition (1) above. But it is possible that abelian varieties satisfying conditions (1)--(3) in the proof of Proposition~\ref{prop:minimal-examples} might exist for sufficiently large $p$. Hence in the light of this remark and Theorem~\ref{thm:characterisation} it seems to me that perhaps  Conjecture of \citep*[Conjecture 4.1]{li10} needs to be modified. In fact, I propose two separate conjectures, depending on whether one fixes the characteristic or one fixes the dimension. Both the conjectures should be true. The fixed dimension version is inspired by \citep*{liedtke09}. I note that Conjecture~\ref{li-mod-dim} replaces \citep*[Conjecture 4.1]{li10}.

\bcon[Fixed Dimension Version]\label{li-mod-dim} Let $d\geq 1$ be a fixed integer. Then there exists an integer $n_0(d)\geq 1$ with the following property: for any smooth, projective variety $X/k$ of dimension $d$ over an algebraically closed field $k$ and with trivial tangent bundle is an abelian variety if $p>n_0(d)$. For $d=1$, $n_0(d)=1$; for $d=2$, one has $n_0(d)=3$ (by Theorem~\ref{thm:surface-case}).
\econ

Before I state the fixed characteristic version, let us make the following elementary observation.

\blem\label{lem:n1} There exists an integer $n_1(p)\geq 0$ with the following property. For any smooth, projective variety $X$ of dimension $d$ over an algebraically closed field of characteristic $p>0$. If $d<n_1(p)$ then $X$ is an abelian variety.
\elem
\bp
Suppose, for a given $p$, there exists a smooth, projective variety $Z$ with trivial tangent bundle which is not an abelian variety. Then for every integer $n\geq \dim(Z)$, there exists a variety $Y$ of this sort with $\dim(Y)=n$. Indeed one may simply take $Y=Z\times E^{n-\dim(Z)}$ for any elliptic curve $E$.  So take a variety $Z$ with the above properties of the smallest dimension and let $n_1(p)=\dim(Z)$. If no such variety $Z$ exists one can simply take $n_1(p)=0$. Then every smooth projective variety $X$ of dimension $\dim(X)<n_1(p)$ is an abelian variety by construction.
\ep
For $p=2,3$, $n_1(p)=2$ by Theorem~\ref{thm:surface-case}. The following is the fixed characteristic version of the conjecture.

\bcon[Fixed Characteristic Version]\label{li-mod-p} Let $p$ be any fixed prime number.
The number $n_1(p)$ constructed in Lemma~\ref{lem:n1} has the property that $n_1(p)\geq 4$ for $p\geq 5$.
\econ

%
\bibliographystyle{plainnat}

\end{document}